\documentclass[11pt]{article}
\usepackage[utf8]{inputenc}
\usepackage[T1]{fontenc}
\usepackage{amsmath}
\usepackage{amsfonts}
\usepackage{amssymb}
\usepackage[version=4]{mhchem}
\usepackage{stmaryrd}
\usepackage{bbold}
\usepackage{graphicx}
\usepackage{overpic}
\usepackage{fullpage}
\usepackage{url}

\newtheorem{theorem}{Theorem}

\newtheorem{Lemma}[theorem]{Lemma}
\newtheorem{Corollary}[theorem]{Corollary}

\newtheorem{Fact}[theorem]{Fact}
\newenvironment{proof}{\paragraph{Proof}}{\hfill$\square$}

\DeclareMathOperator{\Sek}{S}
\DeclareMathOperator{\Tek}{T}
\DeclareMathOperator{\KZ}{KZ}

\title{Network Routing on Regular Digraphs and Their Line Graphs}

\author{Vance Faber and Noah Streib\\
Center for Computing Sciences\\
Bowie, Maryland\\
Revision: March 19, 2023}
\date{}

\begin{document}
\maketitle

\begin{abstract} 
This paper concerns all-to-all network routing on regular digraphs.  In 
previous work we focused on efficient 
routing in highly symmetric digraphs with low diameter for fixed 
degree.  Here, we show that every connected regular digraph has an all-to-all 
routing scheme and associated schedule with no waiting.  In fact, this 
routing scheme becomes more efficient as the diameter goes down with 
respect to the degree and number of vertices.  
Lastly, we examine the simple scheduling algorithm called 
``farthest-distance-first'' and prove that it yields optimal schedules
for all-to-all communication in networks of interest, including Kautz graphs. 
\end{abstract}

\section{Introduction}

In the first two sections of this paper, we provide a general method 
for scheduling an all-to-all communication pattern on regular networks.  
In particular, we show that, for any $d$-regular network of diameter $D$, 
an all-to-all can be scheduled in time
\[
\mu(d, D)=\frac{D d^{D+1}-(D+1) d^{D}+1}{(1-d)^{2}}
\]
and furthermore, once a message has begins its journey through the 
network, it never has to wait at an intermediate node prior to arriving
at its final destination.  The value $\mu(d, D)$ is independent of the
number of nodes in the network, and thus becomes more useful for highly
efficient networks, like those close to Moore's bound~\cite{deg-diam}.
In fact, in those cases, this bound is close to optimal.  In Section~\ref{sec:lg},
we show that we can do even better than $\mu(d, D)$ by taking advantage
of the structure of certain dense graphs.
Lastly, in Section~\ref{sec:FDF}, we show that a local routing algorithm called 
``farthest-distance-first'' also performs optimally for all-to-alls in cases of 
interest, despite the fact that it is has no knowledge of global 
structure of the network.

\subsection{Notation and Terminology}

We begin by defining some general terms and concepts from network routing.

{\bf Network model}. We model a network as a connected directed graph, or 
\emph{digraph}, with vertices representing routers and edges 
representing connections between routers.  For simplicity, we assume 
that the transfer time and capacity of every connection is identical. 
The graph is not necessarily simple---there may be multiple 
connections, virtual or real, between two vertices.  We also assume 
that each router has in-degree and out-degree $d$ (i.e the digraph is 
$d$-\emph{regular}) for some positive integer $d$, and that on a single 
clock tick the router can perform any one-to-one function of inputs to 
outputs. The \emph{distance} between a vertex $u$ and a vertex $v$ is 
the length of the shortest directed path from $u$ to $v$. 

{\bf All-to-all routing}.  We have discussed elsewhere~\cite{DF, CFS} the 
importance of this communication pattern, in which each processor in a 
network needs to send a message to every other processor. The pattern 
shows up, for example, in any iterative program that needs to multiply 
two matrices with columns distributed over processors.  Each processor 
has a block of data it needs to get from every other processor and so 
each processor needs to send a block of data to every other processor. 
A long-sought goal is given a fixed degree and number of processors, 
find a network which gives the best all-to all time. In order compare 
two networks with the same degree but a different number of processors, 
we divide the time for all-to-all by the number of processors. We call 
this the \emph{efficiency} of the network.

{\bf Factors}. A \emph{1-factor} in a directed graph $G=(V,E)$ is 
a subgraph with both in-degree and out-degree equal to 1. (Some authors 
have called this a 2-factor.  Our definition seems more consistent with 
the terminology in undirected graphs, where each vertex in the subgraph 
induced by the edges of a $k$-factor has degree $k$; in particular,
a 1-factor in an undirected graph is a matching.)  It is known 
that, for every regular directed graph with in-degree and out-degree $d$,
the edge set can be partitioned into $d$ 1-factors.  For completeness, 
we give the proof of this fact here. 

\begin{Fact}
Let $G=(V,E)$ be a $d$-regular digraph.  Then $E$ can be partitioned 
into $d$ 1-factors.
\end{Fact}
\begin{proof}  Form an auxiliary graph $B$ with two new vertices $u^{\prime}$ 
and $u^{\prime \prime}$ for each vertex $u$. The edges of $B$ are the 
pairs $\left(u^{\prime}, v^{\prime \prime}\right)$ where $(u, v)$ is a 
directed edge in $G$. The undirected graph $B$ is bipartite and $d$-regular,
and so by Hall's Marriage Theorem, it can be decomposed 
into $d$ 1-factors. Each of these 1-factors corresponds to a directed 
1-factor in $G$.
\end{proof} \\

{\bf Routing schemes}. Let $G=(V,E)$ be a directed graph. A 
\emph{routing task} $R$ is a collection of pairs $(u,v) \in V \times 
V$.  A \emph{routing scheme} $S := S(R)$ for $R$ is a collection of 
directed walks $P(u,v)$ in $G$, one for each pair in $R$, from the 
source $u$ to the destination $v$.  The \emph{dilation} $D$ of $S$ is the 
maximum length of a walk in $S$.  The \emph{congestion} $C$ of $S$ is the 
maximum, over all $e \in E$, of the number of times $e$ appears in $S$.  
A \emph{routing schedule} $\Sigma$ of time $\tau$ assigns a time $t$ 
with $1 \leq t \leq \tau$ to every edge of every walk in $S$ such that
\begin{enumerate}
  \item on every walk, time is increasing,
  \item no two edges have the same time.
\end{enumerate}
The time $\tau$ is sometimes referred to as the \emph{makespan} of $S$.
Clearly, both $D$ and $C$ are lower bounds on $\tau$.

A schedule has \emph{no waiting} if the times along each walk are successive. 
We often represent a schedule as a rectangular array with the rows 
corresponding to edges of $G$ and the $\tau$ columns corresponding to 
time. The entries in the array denote which path is utilizing the given 
edge at the given time.  An entry can be empty if no path is using the 
edge at that time.  See Figure~\ref{fig:sched_example} for an example.
Note that if we have two schemes $\Gamma_{1}$ and 
$\Gamma_{2}$ consisting of disjoint pairs then $\Sigma_{1} \Sigma_{2}$ 
is a schedule of time $\tau_{1}+\tau_{2}$ for their union.

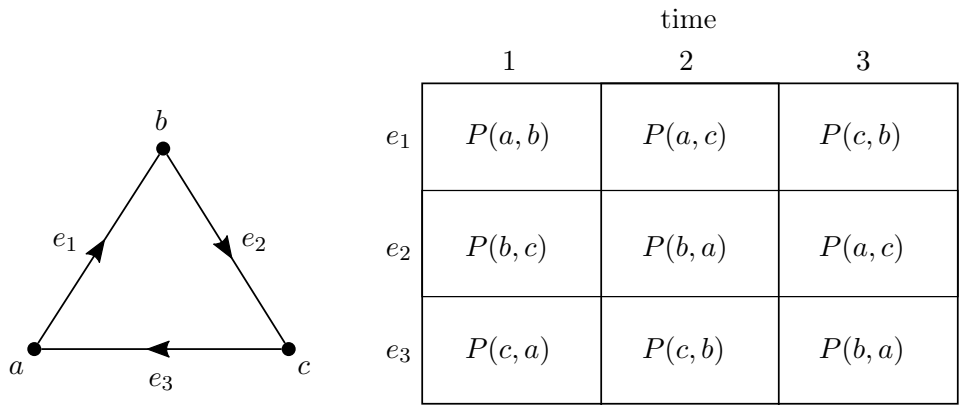
\begin{figure}[!htb]
\begin{center}
\begin{overpic}[width=0.75\textwidth,grid=false,tics=20]{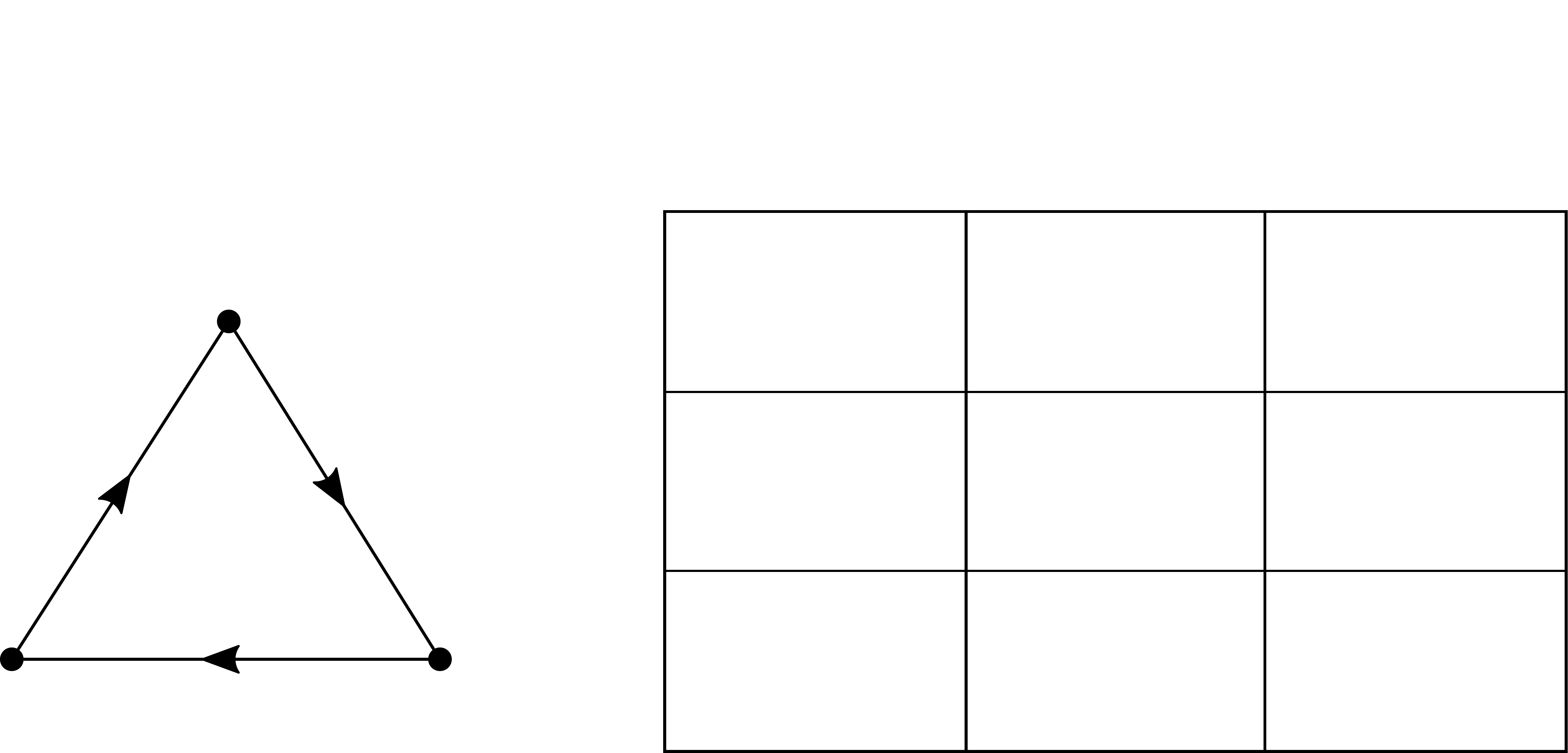}
\put(2.7,17){$e_1$}
\put(23,17){$e_2$}
\put(13,2){$e_3$}
\put(-2,3.3){$a$}
\put(13.7,29.5){$b$}
\put(29,3.3){$c$}
\put(38.5,28){$e_1$}
\put(38.5,16){$e_2$}
\put(38.5,5){$e_3$}
\put(47,28){$P(a,b)$}
\put(47,16){$P(b,c)$}
\put(47,5){$P(c,a)$}
\put(66,28){$P(a,c)$}
\put(66,16){$P(b,a)$}
\put(66,5){$P(c,b)$}
\put(85,28){$P(c,b)$}
\put(85,16){$P(a,c)$}
\put(85,5){$P(b,a)$}
\put(51,36){$1$}
\put(70,36){$2$}
\put(89,36){$3$}
\put(68,40.5){$\text{time}$}
\end{overpic}
\caption{On the right, a rectangular array representing the schedule
of an all-to-all for the graph on the left}
\end{center}
\label{fig:sched_example}
\end{figure}

A \emph{labeling} of $G$ is an assignment of labels 
$0 \leq L \leq \Delta$ to the edges of $G$ such that, for each $v \in 
V$, no two out-edges of $v$ and no two in-edges of $v$ have the same 
label. Suppose $P$ and $Q$ are two paths in a routing schedule for a 
labeled graph $G$ starting from the same vertex. We say $P$ and $Q$ are 
\emph{label-disjoint} if edges assigned the same time have different labels.

\section{Routing in regular digraphs}

An important construction will be the complete directed $d$-ary 
tree of depth $D$, which we denote $T$.

\begin{figure}[!htb]
\begin{center}
\includegraphics[width=0.25\textwidth]{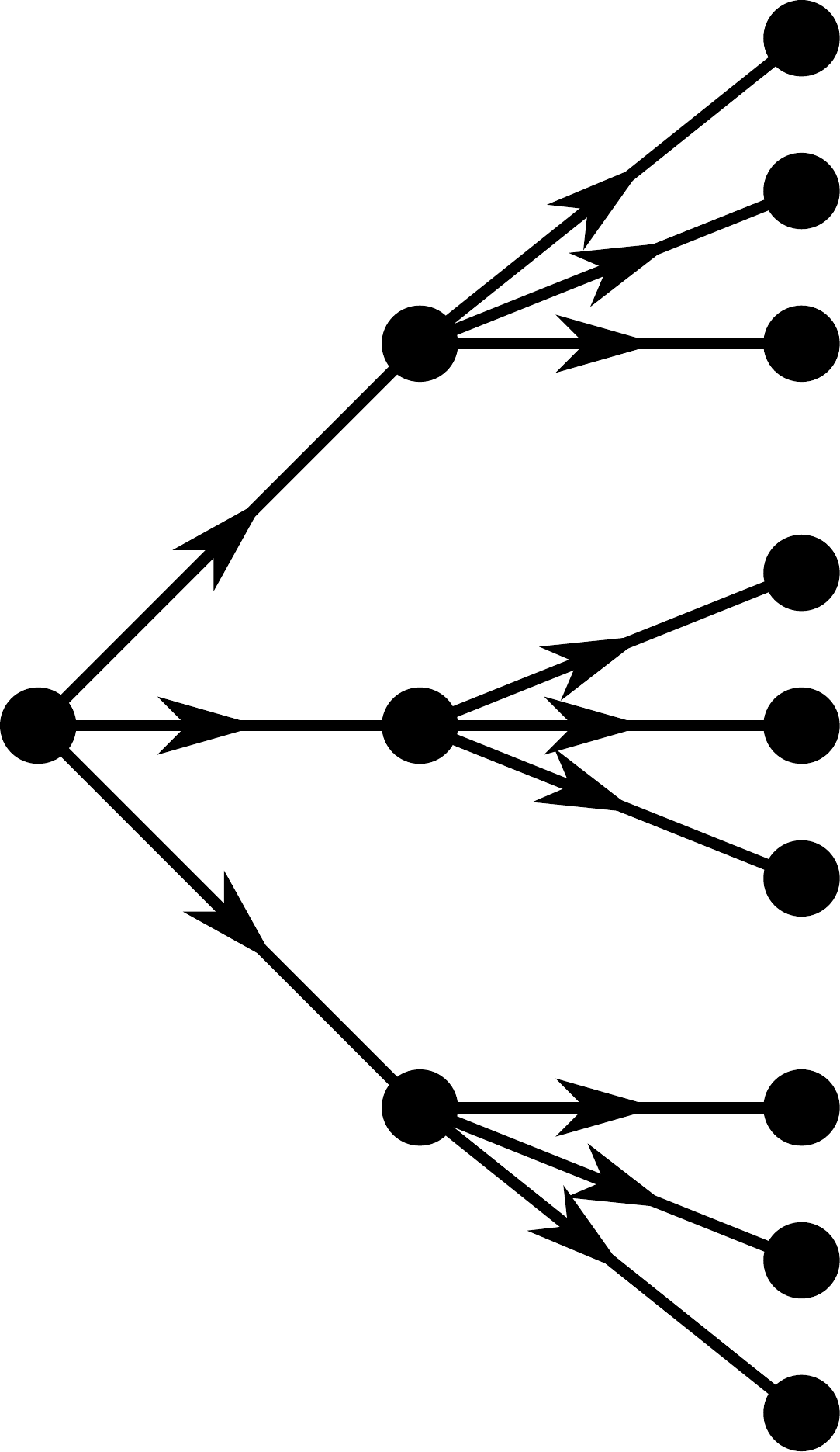}
\caption{A complete $3$-ary tree of depth $2$}
\end{center}
\label{fig:dary}
\end{figure}

{\bf Paths, walks and words}. We differentiate between a \emph{path} which is 
acyclic and a \emph{walk} which is allowed to repeat vertices or edges. We say 
a set of walks $W$ is a \emph{walk cover} of the digraph $G$ if for every pair 
of vertices $u$ and $v$ (where $u$ and $v$ can be equal) there is a unique 
walk from $u$ to $v$ in $W$.  The walks of length $k$ in the walk cover form 
a set $W_{k}(G)$.  We let $K=K(G, W)$ be the set of indices of non-empty 
$W_{k}$.  We also let $w_{k}=\left|W_{k}\right|$. 

Let $G$ be a regular digraph of degree $d$ and let $\left\{F_{i} \mid 0 
\leq i<d\right\}$ be a 1-factorization and $u$ be a vertex. If 
$\omega=\left(F_{c_{0}}, F_{c_{1}}, \ldots, F_{c_{k-1}}\right)$ is a 
\emph{word} of length $k$ (which we will often write just using the 
indices of the 1-factors; i.e. $\omega=\left(c_{0}, c_{1}, \ldots, 
c_{k-1}\right)$) then $u \omega$ describes both a path to a unique 
vertex in the complete $d$-ary tree $T$ and a walk in the graph $G$. 
The set of all these words forms a subtree of $T$. We can label the 
vertices $v$ of distance $k$ from the root $u$ by elements of the 
abelian group $A=\mathbb{Z}_{d}^{k}$ and we assume the elements of 
$A,\left(c_{0}, c_{1}, \ldots, c_{k-1}\right)$, to be colexicographically 
ordered. This ordering is the same ordering we get by identifying a 
vertex with a base $d$ integer $c_{0}+c_{1} d \ldots+c_{k-1} d^{k-1}$ 
and it also provides a labeling of the edges of the tree $T$ by the 
members of the group $\mathbb{Z}_{d}$. Thus there is a 1-1 
correspondence between paths from $u$ and the set of words in the 
$c_{i}$ of length at most $D$.

Let $w=\overrightarrow{1}$ be the element of $A$ corresponding to 
$1+\ldots+d^{k-1}=\frac{d^{k}-1}{d-1}$. Consider the $d$ paths 
$T(i,0)=iw$ for $0 \leq i<d$ (that is, $iw=(i,i,\dots,i))$. Clearly no 
two of these paths share an edge in the tree.  Fix $z$ in the subgroup 
of $A$ consisting of elements of the form $\left(c_{0}, c_{1}, \ldots, 
c_{k-2}, 0\right)$; that is, $0 \leq z<d^{k-1}$. Then the $d$ paths 
$T(i,z)=z+iw$, $0 \leq i<d$, are also label disjoint because the 
difference between any two is a multiple of $w$ (to be clear, the sum 
$z+iw$ is over two elements of $\mathbb{Z}_{d}^{k}$). This proves the 
following lemma.

\begin{Lemma}\label{lem:dtree}
There exists a labeling of the edges of the complete directed 
$d$-ary tree $T$ with root $u$ and of depth $D$ by members of 
$\left\{F_{i} \mid 0 \leq i<d\right\}$ such that the set of all paths 
of length $k$ have a disjoint decomposition into $d^{k-1}$ sets, 
$T(z)=\{T(i,z) \mid 0 \leq i<d\}$ with $0 \leq z<d^{k-1}$, of $d$ 
label-disjoint paths $T(i,z)=z+iw, 0 \leq i<d$.
\end{Lemma}

\begin{Lemma}\label{lem:scheme}
Let $G$ be a regular digraph of degree $d$ with a 
1-factorization $F=\left\{F_{i} \mid 0 \leq i<d\right\}$. Suppose that 
for every vertex $u$, we have a set of distinct words $W(u)$ in $F$ of 
length $k$. Let $S$ be the scheme consisting of the union over 
all $u$ of all the walks $u W(u)$ in $G$. Then there exists a routing 
schedule $\Sigma$ for $S$ with no waiting and time $\tau \leq k 
d^{k-1}$.
\end{Lemma}
\begin{proof}
For each vertex $u$, $W(u)$ is a subtree of the complete $d$-ary 
tree $T(u)$ rooted at $u$. Consider the sets of paths $T_{u}(z)=\left\{T_{u}(i, z) 
\mid 0 \leq i<d\right\}$ in $T(u)$ from Lemma~\ref{lem:dtree} for fixed $z$. We form 
a schedule for each $z$ which is the union of the $T_{u}(z)$ for all 
$u$ where we leave an entry empty if $z$ is a leaf of $T(u) \backslash 
W(u)$. We have to show that this collection of paths meets the 
definition of a schedule by showing that no path $T_{u}(i, z)=z+i w$ 
and $T_{v}(j, z)=z+j w$ have a common edge and time. Let's suppose to 
the contrary that the first common edge occurs at time $t$. Then 
recalling that $z=\left(z_{0}, z_{1}, \ldots, z_{k}\right)$ and $j w=(j, j, 
\ldots, j)$ are vectors, the label of that common edge is both 
$z_{t}+i$ and $z_{t}+j$ which means that $i=j$ in $\mathbb{Z}_{d}$. But 
this means that the edges labeled $z_{t-1}+i$ on $T_{u}(i, z)=z+i w$ 
and $z_{t-1}+j$ on $T_{v}(j, z)=z+j w$, which are both in-edges to the 
same vertex in a 1-factorization, have the same label and therefore 
must be the same edge which contradicts the assumption that the first 
common edge was at time $t$. 
\end{proof}

\begin{theorem}
Let $G$ be a regular digraph 
of degree $d$ with 1-factorization $F$. Suppose there is an increasing 
sequence of lengths $\left(D_{k}\right)$ so that
\begin{enumerate}
  \item for every pair of vertices $(u, v)$ of $G$ there is a word 
  $\omega(u, v)$ in $F$ with length equal to one of the $D_{k}$;
  \item $v$ is on the walk $u W(u, v)$;
  \item for each $u$, the $W(u, v)$ are distinct.
\end{enumerate}
Then there exists a routing schedule $\Sigma$ for $S$, with $S$ as defined
in Lemma~\ref{lem:scheme}, with no waiting 
and time $\tau \leq \sum_{k} D_{k} d^{D_{k}-1}$.
\end{theorem}
\begin{proof}
The theorem follows directly from Lemma~\ref{lem:scheme}. Note that if the walk 
$u W(u, v)$ reaches $v$ before its end, we might choose to leave the 
rest of its schedule blank or we might choose to continue the walk 
until we eventually return to $v$ as the final destination.  We include 
this case as having such longer-than-necessary paths can be reduce the
number of non-zero $D_{k}$, which can be advantageous for scheduling and
bounding purposes.
\end{proof} \\

We call this type of routing a \emph{regular routing}.
By choosing $D_{k}=k$ for $1 \leq k \leq D$ we are guaranteed to cover 
every shortest path in a connected regular graph of diameter $D$. This gives us 
the following corollary.

\begin{Corollary}
For every connected regular digraph of degree $d$ and diameter $D$ there is a 
regular routing of time
\[
\tau \leq \mu(d, D)=\frac{D d^{D+1}-(D+1) d^{D}+1}{(1-d)^{2}}.
\]
\end{Corollary}
\begin{proof}
For each vertex $u$, there is a set of shortest paths of length $k \leq 
D$ which form a tree $T(u)$. Thus, by the theorem, these 
paths yield a regular routing of time
\[
\tau \leq \sum_{k=1}^{D} k d^{k-1}=\sum_{k=0}^{D-1}(k+1) d^{k} .
\]
However
\[
\sum_{k=0}^{D-1}(k+1) x^{k}=\frac{d}{d x} \sum_{k=0}^{D} x^{k}=\frac{d}{d x} \frac{x^{D+1}-1}{x-1}=\frac{D x^{D+1}-(D+1) x^{D}+1}{(x-1)^{2}} .
\]
\end{proof} \\

In other words, $\mu(d, D)$ is an upper bound on the makespan of an 
all-to-all for every connected $d$-regular digraph of diameter $D$, 
and moreover we can produce such a schedule with no waiting.  

\subsection{Bounds for dense graphs}

Let $G$ be a connected $d$-regular digraph of diameter $D$ with $n$ 
vertices.  Given a vertex $u$, let $N_k(u)$ be the set of vertices 
distance $k$ from $u$ in $G$, and let $n_k(u) = |N_k(u)|$.  Thus, the 
number of edges used in paths from $u$ to the vertices in $N_k(u)$ in 
any all-to-all routing scheme is at least $kn_k(u)$.  Since the graph 
has degree $d$, note that $n_{k+1}(u) \leq dn_k(u)$ and $n \leq M(d,D)$,
where 
\[
M(d,m) = \sum_{k=1}^m d^k = \frac{d^{m+1}-1}{d-1}.
\]
We say that $G$ is \emph{dense} if $M(d,D-1) < n \leq M(d,D)$, and the
\emph{density} of $G$ is $n/M(d,D)$.  Notice that $n$ must be at most 
$M(d,D)$; graphs achieving this bound are called directed Moore graphs, 
although no such graphs are known to exist for non-trivial $d$ and $D$.

\begin{theorem}\label{thm:dense}
Let $(G_n)$ be a sequence of dense $d$-regular graphs of diameter $D$
with increasing size, where $G_n$ has $n$ vertices.  As $n$ increases,
the ratio of the optimal makespan for an all-to-all to $\mu(d,D)$ approaches 1.
\end{theorem}
\begin{proof}
A lower bound on the makespan for an all-to-all of $G_n$ can be obtained
by assuming that $|P(u,v)| = k$ for all vertices $u$ and $v \in N_k(u)$,
and by assuming that we have a schedule that uses every edge at every
time step.  That is, the makespan is at least $\sigma$, where
\[
\sigma = \frac{1}{nd} \sum_{u \in V} \sum_{k=1}^{D} kn_k(u)
       \geq \frac{1}{nd} \sum_{u \in V} \left( \sum_{k=1}^{D-1} kd^{k-1} + D(n-M(d,D-1)) \right)
       = \sum_{k=1}^{D-1} kd^{k-1} + \frac{D}{d}\left(n-M(d,D-1)\right)
\]
where the inequality uses the assumption that the graph is dense.  After
rewriting $\mu(d,D)$ in the form $\sum_{k=1}^{D-1} kd^{k-1} + \frac{D}{d}d^D$,
we can then deduce
\[  
\frac{\sigma}{\mu} = 1 + \frac{\sigma-\mu}{\mu} = 1 + \frac{D(n-M(d,D))}{d\mu}.
\]
Clearly this quantity goes to 1 as $n$ approaches $M(d,D)$.
\end{proof}\\

Theorem~\ref{thm:dense} shows that our regular routing scheme is 
approximately optimal for dense $d$-regular directed graphs.  As these
graphs provide the structure for the construction of highly efficient 
networks, this scheme should be useful in practice.

\section{Routing on line graphs}\label{sec:lg}

The website~\cite{deg-diam} maintains a list of the highest-density 
$d$-regular graphs of diameter $D$ that we know of, for all pairs of 
positive integers $(d,D)$.  At this moment in time, the family of Kautz 
graphs occupies every entry for $d>2$, and because of this, Kautz 
graphs are popular in the computer network community. There is work 
demonstrating a practical approach to their construction as 
well~\cite{wadi}, and hence machines with Kautz interconnection 
networks have been built and sold for practical use~\cite{sicortex}.

One way to define Kautz graphs is as follows (we will provide another 
definition in Section~\ref{sec:FDF}).  Let $\KZ(d,1)$ be the bidirected 
complete graph on $d+1$ vertices.  Then $\KZ(d,D)$ can be defined 
inductively as $\KZ(d,D) = L(\KZ(d,D-1))$ for $D\geq2$, where by $L(G)$ 
we mean the directed line graph of $G$.  That is, given a digraph 
$G=(V,E)$, $L(G)$ has a vertex for every edge $(u,v) \in E$ and a 
directed edge between vertices $(u,v)$ and $(v,w)$.

The vertex set of $\KZ(d,D)$ has size $(d+1)d^{D-1}$.  Thus, $\KZ(d,D)$ 
is dense, as defined in the previous section.  Namely, 
\[
(d+1)d^{D-1} = d^D + d^{D-1} > d^D > \frac{d^D-1}{d-1} = M(d,D-1)
\]
Hence our regular routing scheme is close to optimal for $\KZ(d,D)$.
However, as we will now show, we can do a bit better.

\begin{theorem}
If $G$ has a walk cover with indices from the set $K$, then $L(G)$ has 
a walk cover with indices from the set $K+1$. 
\end{theorem}
\begin{proof}
Let $e=(u,v)$ and $f=(w,x)$ be two distinct 
vertices in $L(G)$ ($v$ and $w$ might be equal). Then for some $k$ there 
is a walk $P$ in $G$ from $v$ to $w$ with $k$ edges. Then the walk $e P 
f$ has $k+2$ vertices in $L(G)$ and thus is a walk of length $k+1$. Now 
suppose $u=w$ and $v=x$ so that $e=f$. In $G$ there is a walk $P$ with 
$k$ edges from $v$ to $u$ so that again $e P e$ is a walk with $k+1$ 
edges from $e$ to $e$.
\end{proof} \\

{\bf Application 1: Kautz Graphs}. Consider $\KZ(n-1,1)$, the bidirected 
complete graph on $n$ vertices, and let $K=\{0,1\}$. Notice that we 
have to include 0 because there are no loops and hence no paths of 
length 1 from a vertex to itself. But now by induction, the iterated 
line graph of degree $n-1$ is covered by sets with indices $\{D-1, D\}$ 
where $D$ is the diameter.  Using this walk cover to define the routes
for an all-to-all, we see that we can use the parts of the schedule for the 
regular routing restricted to paths of lengths $D-1$ and $D$; indeed,
each time step in the regular routing is devoted to paths of a common
length.  Thus, we can schedule an all-to-all for $\KZ(d,D)$ in time
$\sum_{k=D-1}^{D} kd^{k-1}$.

\begin{Lemma}\label{lem:KZdense}
If $G$ has degree $d$ and diameter $D$ and more than $(d+1) d^{D-1}$ 
vertices, any cover has more than 2 indices.
\end{Lemma}
\begin{proof}
Let $F_{1}, F_{2}, \cdots, F_{d}$ be a 1-factorization of $G$. Then 
each walk of length $k$ starting from a fixed vertex $v$ is a word of 
cardinality $k$ in the $F_{i}$. The number of distinct words of 
cardinality $k$ is $d^{k}$ so the largest number of words using two $k$ 
is $d^{D-1}+d^{D}=(d+1) d^{D-1}$.
\end{proof} \\

{\bf Application 2: The Alegre Graph Family}. The Alegre 
graph~\cite{alegre} has 25 vertices and is the densest $2$-regular 
graph with diameter $4$ that is known.  By iteratively taking line 
graphs, we produce the densest $2$-regular graphs of diameter $D \geq 
4$ that are known~\cite{deg-diam}.  The Alegre graph and is covered by 
the set with indices $K=\{0,3,4\}$. This is a cover of minimum 
cardinality by Lemma~\ref{lem:KZdense}.  Every iterated line graph also 
has a minimal cover of cardinality 3.  Thus, for this family of graphs, 
we can provide schedules for all-to-all that are a bit better than the 
general regular routing, as we did for the family of Kautz graphs.

\begin{theorem}
Let $\left\{W_{k} \mid k \in K\right\}$ be a walk cover of $G$. Then in 
the associated walk cover of $L(G)$, $w_{k+1}(L)=d^{2} w_{k}(G)$.
\end{theorem}
\begin{proof}
For each walk $P(v, w)$ in $W_{k}(G)$ and edges $u v$ and $w x$, the 
walk $(u v) P(v, w)(w x)$ in $L$ has length $k+1$. There are $d^{2}$ 
such $u$ and $x$.
\end{proof} \\

{\bf Observation}. Note that if $G$ has $n$ vertices then 
$\sum_{k \in K} w_{k}=n^{2}$ and $\sum_{k \in K} w_{k+1}(L)=\sum_{k \in 
K} d^{2} w_{k}=(d n)^{2}$, as expected.

\section{Farthest-Distance-First Routing}\label{sec:FDF}

Given a routing scheme $S$ on a digraph $G$, there are many ways of
constructing a routing schedule. If one knows the 
topology of $G$ and structure of $S$, then it is often the case that 
efficient schedules can be constructed, as we have seen above.  However, 
these schedules can suffer in practice from not being flexible 
enough to deal with practical impediments like broken links or the 
lack of global synchronicity.

An alternative approach would be to use a routing algorithm that 
applies to any given $G$ and $S$. The efficiency of the algorithm can 
be evaluated in terms of the congestion and dilation, and recall that 
both $C$ and $D$ are lower bounds on the length of any schedule. In 
their seminal work~\cite{LMR}, Leighton-Maggs-Rao showed that $O(C + 
D)$-time schedules always exist, although at the time their method was 
non-constructive, as there was no known algorithmic version of the 
Lov\'{a}sz Local Lemma. Since then, their solution has been made both 
algorithmic~\cite{Fast-LMR, harris} and simpler~\cite{Simple-LMR}.

The algorithm of~\cite{LMR} is optimal up to a constant multiplicative 
factor. However, this algorithm may not be preferred if one is 
concerned about the magnitude of this multiplicative factor, nor if 
computing an efficient schedule needs to be done very quickly.  In this 
section we address both of these issues.  We describe conditions on $S$ 
such that, if satisfied, schedules of length $C$ can be achieved with a 
simple local routing strategy, and without the need to precompute a 
global schedule at all\footnote{FDF routing can also produce optimal 
schedules when the optimal value is greater than $C$, but here we are 
more concerned with the case that $C$ is the right answer.}.

The routing scheme analyzed here is \emph{farthest-distance-first}, or FDF. It 
acts according to one simple rule. At any moment in time, each edge $e$ 
of $G$ has a buffer that stores the jobs waiting for their next hop to be 
executed on $e$. The rule is: $e$ executes the job with the most remaining 
hops; ties are broken arbitrarily. It is known that FDF performs 
optimally when $G$ is in a certain class of trees~\cite{leung_book} (all 
in-degrees are 1 or all out-degrees are 1), although, if $G$ is allowed 
to be any directed tree, there are $S$ for which the FDF schedules can be 
arbitrarily poor~\cite{FDF_peis}. Here, we expand the list of instances $(G,S)$ for 
which FDF performs optimally.

\subsection{Main FDF Theorem}

As input, we are given the pair $(G,S)$, where $G$ is a $d$-regular 
network.  Note that we allow $p_1,p_2 \in S$ with $p_1 = p_2$, and 
furthermore we need not assume that paths are shortest paths in $G$. For 
an edge $e \in G$, let $\Sek(e,k)$ be the set of paths $p \in S$ that use $e$ as the 
$k^\text{th}$-to-last edge (i.e. just prior to traversing $e$, $p$ has $k$ hops 
remaining). Let $\Tek(e,k)$ be the last time slot in the FDF schedule of 
$(G,S)$ in which $e$ is used as the $k\text{th}$-to-last edge in a task. We 
define $\Tek(e,D+1)$ to be $0$, where we recall that $D$ is the length of 
the longest path in $S$.

Consider the following two properties of a pair $(G,S)$.

\begin{enumerate}
\item There exists an edge $e$ such that  
\[ \sum_{k=i}^D |S(e,k)| \geq \sum_{k=i}^D |S(f,k)| \]
for all edges $f \in N$ and all $1\leq i\leq D$.  We call such an edge
\emph{dominant}.
\item Let $e$ be a dominant edge.  Then the inequality 
\[ d\sum_{k=1}^{i} |S(e,k)| \geq d\sum_{k=1}^{i} |S(f,k)| - |S(f,i)| \]
holds for all edges $f \in N$ and all $1\leq i\leq D$.  Note that for $i=D$,
this inequality is implied by Property 1.
\end{enumerate}

We are now ready to state our main theorem.

\begin{theorem} \label{thm:FDFmain}
Assume $(G,S)$ satisfies Properties 1 and 2 above with $e$ a dominant edge.  
Then FDF produces an optimal schedule that uses $T(e,1)=C$ ticks.
\end{theorem}
\begin{proof}
We proceed by induction on $D$.  The
case $D=1$ is straightforward, so assume $D > 1$.

Delete the last hop from each $p \in S$ to get a new list of jobs $S'$.  
The dilation $S'$ is $D-1$, so we may apply the induction
hypothesis to $S'$.  Note that the schedule we get for $S'$ is identical 
to the schedule we get for $S$ when we restrict out attention to the tasks 
in $S'$; indeed, FDF gives priority to hops in $S'$ over all jobs in 
$S\setminus S'$.  So to recover the FDF schedule for $S$, we need only to 
consider the placement of the final hops of each job.

By induction, the length of the schedule for $S'$ is $\Tek(e,2) = \sum_{k=2}^D |\Sek(e,k)|$.
Thus, all final-hops for $e$ must be waiting in $e$'s out-buffer at time 
$\Tek(e,2)$, as every edge $f$ in $G$ has completed its next-to-
last hops by that time; in particular, $\Tek(f,k) \leq \Tek(e,2)$ for all
$2 \leq k \leq D$.  So the schedule for $S$ will satisfy 
$\Tek(e,1) = \sum_{k=1}^D |\Sek(e,k)|$.  Note then that $\Tek(e,1)=C$,
since $e$ is dominant.

It remains to show that there is no $f$ for which $\Tek(f,1) > \Tek(e,1)$.  
By way of contradiction, suppose that such an $f$ exists.  Since $e$ is
dominant, it must be the case that the schedule for $f$ contains a hole.
That is, there is a time $t < \Tek(f,1)$ for which $f$ is not executing a task.  
Pick $t$ largest such that this is true.  This implies that
at every time after $t$, $f$ is executing some task whose 
previous hop was on a non-$f$ edge; indeed, if there was a task starting
at $f$ that was still waiting to execute its first hope at time $t$, 
then an FDF schedule would not get $f$ a whole at time $t$.  In other
words:

\begin{Fact}\label{f1}
At every time between $t+1$ and $\Tek(f,1)$, there is a job
using $f$ that first used an incident edge $g_j$, where $1 \leq j \leq d$
and $g_1, g_2, \dots, g_d$ are the edges whose head is the tail of $f$.
\end{Fact} 
 
Our goal will be to show that some $g_j$ contradicts the dominance of $e$.
This is clearly the case if $t>\Tek(e,1)$, so we may assume otherwise.
Find the $1 \leq i \leq D$ such that $\Tek(e,i+1) < t \leq \Tek(e,i)$.
Notice the following:

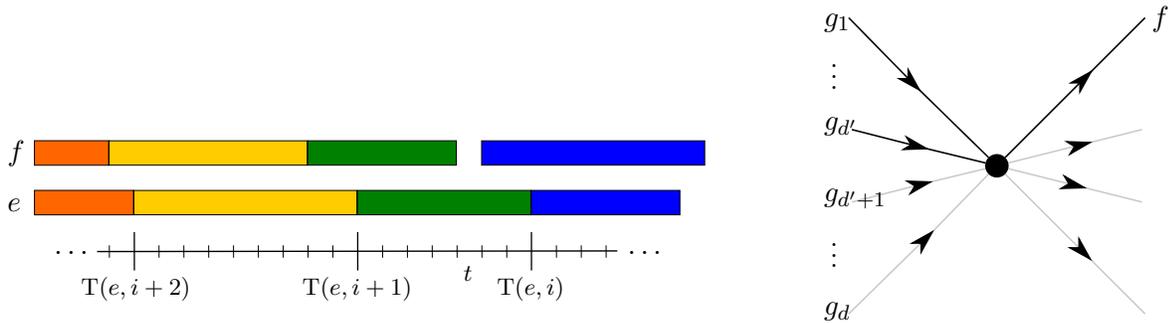
\begin{figure}[!htb]
\begin{minipage}[t]{0.6\linewidth}
\centering
\begin{overpic}[width=0.9\textwidth]{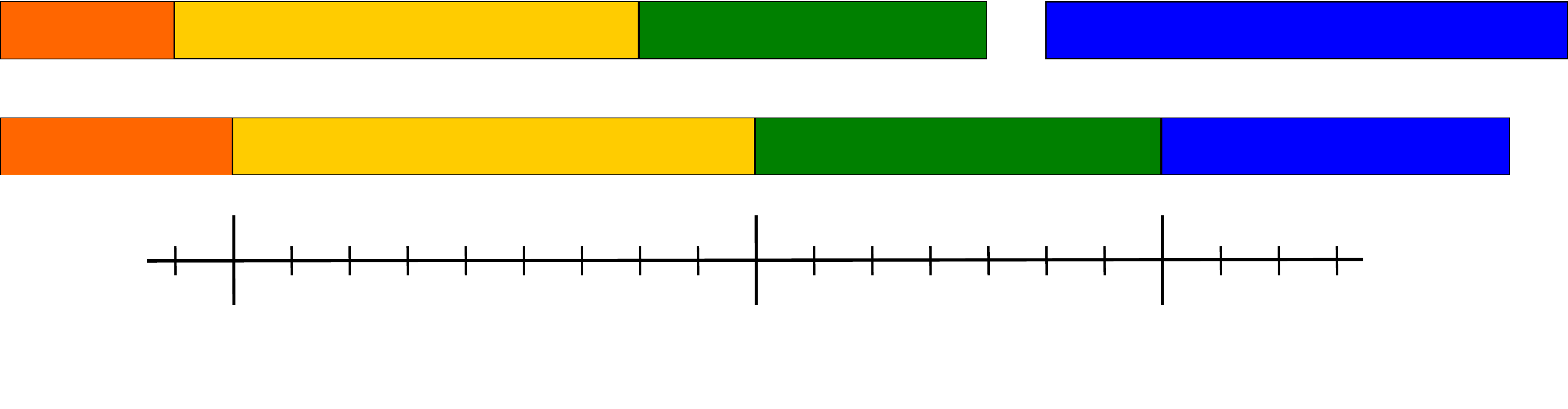}
\put(-4,23){$f$}
\put(-4,15.5){$e$}
\put(7,3){\footnotesize{$\Tek(e,i+2)$}}
\put(40,3){\footnotesize{$\Tek(e,i+1)$}}
\put(69,3){\footnotesize{$\Tek(e,i)$}}
\put(64,5){\footnotesize{$t$}}
\put(3,9){$\dots$}
\put(88.5,9){$\dots$}
\end{overpic}
\end{minipage}
\begin{minipage}[b]{0.4\linewidth}
\centering
\begin{overpic}[width=0.6\textwidth]{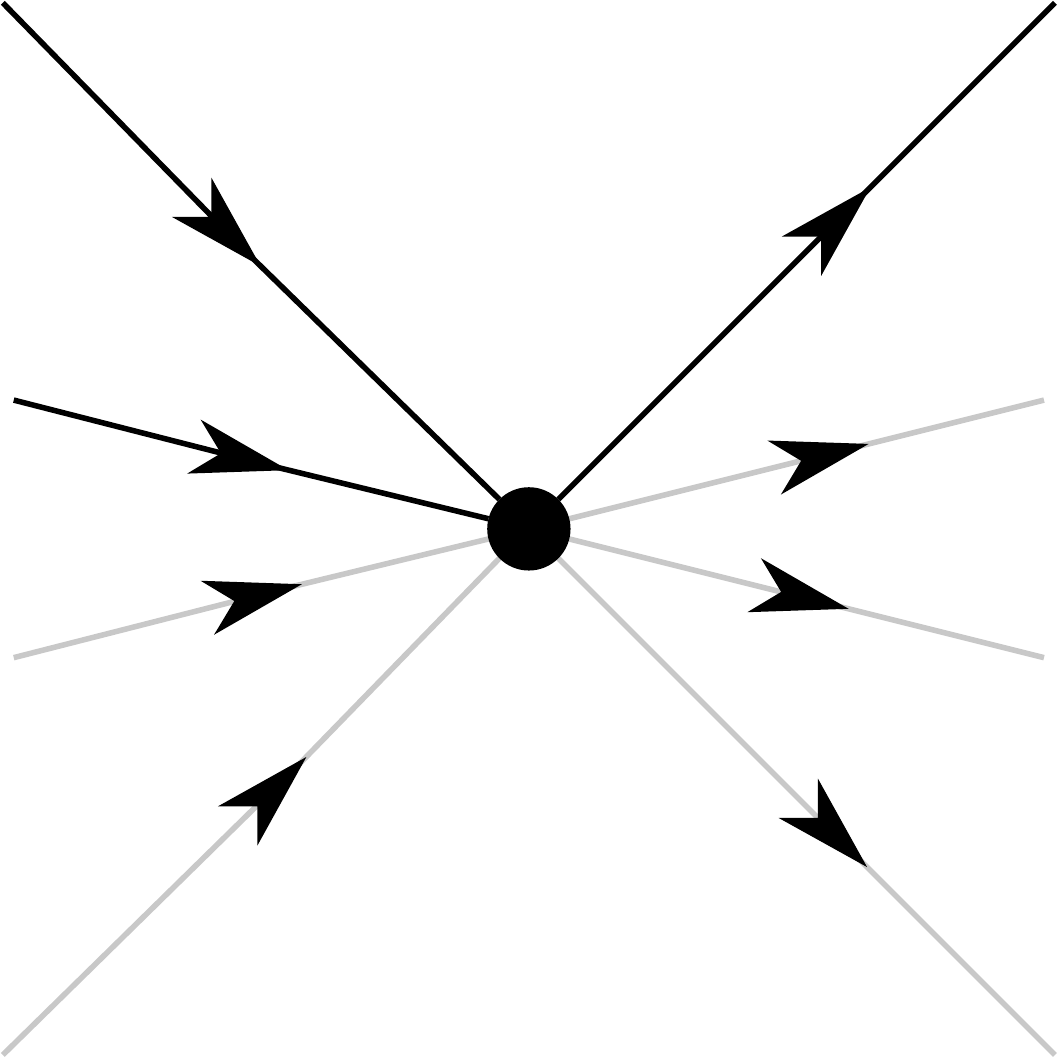}
\put(102,97){$f$}
\put(-8,97){$g_1$}
\put(-6,76){$\vdots$}
\put(-8,62){$g_{d'}$}
\put(-8,38){$g_{d'+1}$}
\put(-6,16){$\vdots$}
\put(-8,0){$g_d$}
\end{overpic}
\end{minipage}
\caption{These figures are meant to help illustrate the definitions in the proof.}
\end{figure}

\begin{Fact}\label{f2}
For every $g_j$ such that there is a $p \in S$ with $p \in \Sek(g_j,i) 
\cap \Sek(f,i-1)$, it must be the case that $g_j$ executed tasks in $\Sek(g_j,i) \setminus \Sek(f,i-1)$
between times $\Tek(e,i+1)+1$ and $t$, as otherwise $f$ would not have a hole 
at time $t$.  Indeed, such a $g_j$ cannot itself have a hole in its
schedule between times $\Tek(e,i+1)+1$ and $t$, as the induction hypothesis
implies that it has received at jobs in $\Sek(g_j,i)$ by time $\Tek(e,i+1)$.
Let $d' \leq d$ be the number of such $g_j$, and without 
loss of generality we will assume these edges are $g_1, g_2, \dots, g_{d'}$.
\end{Fact}

\begin{Fact}\label{f3}
By the induction hypothesis, $\Tek(g_j,i+1) \leq \Tek(e,i+1)$, and hence
all tasks executed on $f$ after time $t$ are in $\cup_{k=1}^{i-1} \Sek(f,k)$.
\end{Fact}

Let $\Sek^*(g_j,i)$ be the subset of $\Sek(g_j,i)$ scheduled at time $\Tek(e,i+1)+1$
or later.  We can now make the following calculation:
\begin{align}
\sum_{j=1}^{d'} |\Sek^*(g_j,i)| & \geq d'(t - \Tek(e,i+1))+ |\Sek(f,i-1)| \\
	& = d'(\Tek(f,1) - \sum_{k=1}^{i-1} |\Sek(f,k)| - \Tek(e,i+1)) + |\Sek(f,i-1)| \\
	& > d'\Tek(e,1) - d'\Tek(e,i+1) - d'\sum_{k=1}^{i-2} |\Sek(f,k)| - (d'-1)|\Sek(f,i-1)|\\
	& = d'(\Tek(e,i+1) + \sum_{k=1}^{i} |\Sek(e,k)|) - d'\Tek(e,i+1) - d'\sum_{k=1}^{i-2} |\Sek(f,k)| - (d'-1)|\Sek(f,i-1)|\\
	& = d'|\Sek(e,i)| + d'\sum_{k=1}^{i-1} |\Sek(e,k)| - (d'-1)\sum_{k=1}^{i-1} |\Sek(f,k)| - \sum_{k=1}^{i-2} |\Sek(f,k)|\\
	& = d'|\Sek(e,i)| + d'\sum_{k=1}^{i-1} |\Sek(e,k)| - d'\sum_{k=1}^{i-1} |\Sek(f,k)| + |\Sek(f,i-1)|
\end{align}    
Inequality (1) follows from Fact~\ref{f1} (for the second term) and Fact~\ref{f2}
(for the first term).  Equation (2) follows from Fact~\ref{f3} and the 
choice of $t$.  The strict inequality (3) uses the assumption that
$\Tek(f,1) > \Tek(e,1)$.  Equation (4) follows from $e$ being dominant.
Everything else is rearranging terms.

Thus, we get that
\[ \sum_{j=1}^{d'} |\Sek^*(g_j,i)| > d'|\Sek(e,i)| \]
whenever
\[ d'\sum_{k=1}^{i-1} |\Sek(e,k)| \geq d'\sum_{k=1}^{i-1} |\Sek(f,k)| - |\Sek(f,i-1)| \]
As Property 2 implies implies this latest relation, we see that there 
must be some $1 \leq j \leq d'$ such that $|\Sek^*(g_j,i)| > |\Sek(e,i)|$.  This
contradicts the strong inductive hypothesis, which finishes the proof.
\end{proof}

\subsection{Application of FDF to all-to-all routing}

Properties 1 and 2 may seem hard to verify in general, but there is at 
least one case for which their verification is quite simple. Theorem~\ref{thm:FDFmain}
applies to all $(G,S)$ that satisfy $|\Sek(e,k)| = |\Sek(f,k)|$ for all 
edges $e,f \in G$. In this case, every edge is dominant. For highly 
symmetric networks---like those that are vertex and edge transitive---and 
highly symmetric communication patterns---like all-to-all---it can 
be the case that $S$ can be chosen to make every edge dominant. For 
example, it is not hard to show that if $G$ is a two-dimensional square 
torus where the dimensions are odd, such an $S$ can be produced from 
shortest paths for an all-to-all.

We can apply Theorem~\ref{thm:FDFmain} to more interesting cases as well.
Recall from Section~\ref{sec:lg} that the family of Kautz graphs 
$\KZ(d,D)$ have walk covers from the set $\{D-1,D\}$.  Let $S$ be the
all-to-all whose routes are defined by these walks.  
Notice that these routes seem wasteful at first glance; 
certainly there are shorter routes we could have used, and jobs using 
these walks occasionally reach their eventual destination before the 
walk has terminated, only to come back later.  However, the congestion 
for $S$ is known to often be better than it would have been had $S$ been  
constructed from shorted paths~\cite{KZ-uniform-1, KZ-uniform-2}. 

Furthermore, we can apply 
Theorem~\ref{thm:FDFmain} to $(\KZ(d,D),S)$. Fix any $v \in G$. The 
routes in $S$ starting at $v$ can be seen as paths from the root of the 
complete $d$-ary tree to all vertices at distance $D-1$ and $D$ from 
the root. Consider $|\Sek(e,k)|$ 
for any edge $e$. If $e$ is part of a walk of length $D$ in 
$k^{\text{th}}$-to-last position, then there are $D-k$ edges proceeding 
it and $k-1$ edges following it. As any of the $d^{D-k}$ walks leading 
up to $e$ exists in exactly one of the $d$-ary trees, and any of these 
walks can be completed to a walk of length $D$ in $d^{k-1}$ ways, we 
see that there are $d^{D-1}$ such walks. Similarly, if $e$ is part of a 
walk of length $D-1$ in $k^{\text{th}}$-to-last position, then there 
are $D-k-1$ edges proceeding it and $k-1$ edges following it, and thus 
there are $d^{D-2}$ such walks. Therefore, $|\Sek(e,k)| = d^{D-1} + 
d^{D-2}$ for all $k < D$. Since $e$ cannot be the 
$D^{\text{th}}$-to-last edge on a walk of length $D-1$, we find 
$|\Sek(e, k)| = d^{D-1}$ when $k = D$. As these calculations are 
independent of $e$, we see that every edge is indeed dominant.

\bibliography{routing}

\end{document}